\documentclass[leqno]{amsart}
\usepackage{amsmath}
\usepackage{amssymb}
\usepackage{amsthm}
\usepackage{enumerate}
\usepackage[mathscr]{eucal}
\theoremstyle{plain}
\newtheorem{theorem}{Theorem}[section]

\theoremstyle{definition}

\newtheorem{remark}[theorem]{Remark}

\newtheorem{example}[theorem]{Example}

\newtheorem{cor}[theorem]{Corollary}
\theoremstyle{remark}

\begin{document}

\title[ Numerical radius inequalities ]{Numerical radius inequalities and its applications in estimation of zeros of polynomials}

\author{Pintu Bhunia, Santanu Bag and Kallol Paul}

\address{(Bhunia)Department of Mathematics, Jadavpur University, Kolkata 700032, India}
\email{pintubhunia5206@gmail.com}

\address{(Bag)Department of Mathematics, Vivekananda College For Women, Barisha, Kolkata 700008, India}
\email{santanumath84@gmail.com}

\address{(Paul)Department of Mathematics, Jadavpur University, Kolkata 700032, India}
\email{kalloldada@gmail.com}

\thanks{ The first author would like to thank UGC, Govt.of India for the financial support in the form of junior research fellowship.}



\date{}
\maketitle
\begin{abstract}
 We present some upper and lower bounds for the numerical radius of a bounded linear operator defined on complex Hilbert space, which improves on  the existing upper and lower bounds. We also present an upper bound for the spectral radius of sum of product of $n$ pairs of operators. As an application of the results obtained, we provide a better estimation for the zeros of a given  polynomial.

\end{abstract}

\section{Introduction}
\noindent Let $B(H)$ denote the $C^*$-algebra of all bounded linear operators on a complex Hilbert space $H$ with usual inner product $\langle .,. \rangle$. Let $T\in B(H)$ and  $W(T)$, $w(T)$, $m(T)$, $\|T\|$ be the  numerical range, numerical radius, crawford number, operator norm of $T$ respectively, defined as follows:
\begin{eqnarray*} 
W(T)&=&\{ \langle Tx,x \rangle ~~:~~ x \in H, \|x\|=1\},\\
w(T)&=&\sup\{|\lambda| : \lambda\in W(T) \},\\
m(T)&=&\inf\{|\lambda| : \lambda\in W(T) \},\\  
\|T\|&=&\sup\{\|Tx\|~~:~~ x \in H, \|x\|=1\}.
\end{eqnarray*}
It is well known that $w(.)$ is a norm on $B(H)$, which is equivalent to the usual operator norm $\|.\|$ and satisfies the inequality 
\[ \frac{1}{2}\|T\|\leq w(T)\leq \|T\|. \]

The first inequality becomes an equality if $T^2=0$ and the second inequality becomes an equality if $T$ is normal. Various numerical radius inequalities improving this inequality have been given in \cite{AAAB,BK,FK,FK2,TY}.  $T$ can be represented as $T=\textit{Re}(T)+i\textit{Im}(T)$, the Cartesian decomposition, where $\textit{Re}(T)$ and $\textit{Im}(T)$ are real part of $T$ and imaginary part of $T$ respectively, i.e., $\textit{Re}(T)=\frac{T+T^*}{2}$  and $\textit{Im}(T)=\frac{T-T^*}{2i}$,  $T^*$ denotes the adjoint of $T$. It is well known  that $w(T)=\sup_{\theta \in \mathbb{R}}\|H_\theta\|$, where  $H_\theta=\textit{Re}(e^{i\theta}T)$. Let $\rho(T)$ be the spectral radius of $T$, i.e., $\rho(T)=\sup\{|\lambda|~~:~~\lambda \in \sigma(T)\}$, where $\sigma(T)$ denotes the point spectrum of $T$. Also it is well known that $\sigma(T) \subseteq \overline{W(T)}$. 

In this paper we obtain an upper bound for the numerical radius of a bounded linear operator which improves on the existing upper bound given in \cite{OK}. Also we obtain a lower bound for the numerical radius of a bounded linear operator which improves on the existing lower bound given in \cite{FK}. We present an  upper bound of the numerical radius in terms of $\|H_\theta\|$ and a lower bound of the numerical radius in terms of spectral values of $\textit{Re}(T)$ and $\textit{Im}(T)$, which improves on existing lower bounds. We also estimate the spectral radius of sum of product of $n$ pairs of operators.  As an application of the numerical radius inequalities obtained here we estimate the zeros of a polynomial. Various mathematicians 
have estimated the zeros of polynomials over the years using different approaches. We show with numerical examples that the estimations obtained by us is better than the existing ones done by  \cite{PB1,PB2}. 
  
\section{On upper bound of numerical radius inequalities}
We begin this section with the following inequality which improves on  upper bound of the numerical radius of a bounded linear operator on complex Hilbert space.

\begin{theorem}\label{theorem:upperbound}
Let $T \in B(H)$. Then  
\[w^{4}(T)\leq \frac{1}{4}w^2(T^2)+\frac{1}{8}w(T^2P+PT^2)+\frac{1}{16}\|P\|^2,\]  where $P=T^{*}T+TT^{*}.$
\end{theorem}

\begin{proof}
We know that $w(T)=\sup_{\theta \in \mathbb{R}}\|H_\theta\|$ where $H_{\theta}=\textit{Re}(e^{i \theta} T)$. Then,
\begin{eqnarray*}
    H_{\theta} & = & \frac{1}{2}(e^{i\theta} T + e^{-i\theta} T^*) \\
   \Rightarrow 4 {H_{\theta}}^2  &= & e^{2i\theta} T^2 + e^{-2i\theta} {T^*}^2 + P \\ 
	\Rightarrow 16 {H_{\theta}}^4  &= & \big(e^{2i\theta} T^2 + e^{-2i\theta} {T^*}^2 + P \big) \big(e^{2i\theta} T^2 + e^{-2i\theta} {T^*}^2 + P \big)\\
	       &= & \big(e^{2i\theta} T^2 + e^{-2i\theta} {T^*}^2 \big)^2+\big(e^{2i\theta} T^2 + e^{-2i\theta} {T^*}^2 \big)P \\ &+& P \big(e^{2i\theta} T^2 + e^{-2i\theta} {T^*}^2 \big)+P^2\\
& = & 4\big(\textit{Re}(e^{2i\theta} T^2)\big)^2+ 2 \textit{Re}(e^{2i\theta} (T^2P+PT^2))  +P^2\\	
\Rightarrow \| {H_{\theta}}^4 \| &\leq &	\frac{1}{4}\big \|	\textit{Re}(e^{2i\theta} T^2)	\big \|^2+\frac{1}{8}\big \|	\textit{Re}(e^{2i\theta} (T^2P+PT^2))	\big \| +\frac{1}{16} \|P\|^2
\end{eqnarray*}
 Now taking the supremum over $\theta \in \mathbb{R}$ in the above inequality we get,
\begin{eqnarray*}
 \Rightarrow w^{4}(T)&\leq&\frac{1}{4}w^2(T^2)+\frac{1}{8}w(T^2P+PT^2)+\frac{1}{16}\|P\|^2.
\end{eqnarray*}
\end{proof}
\begin{remark} 
 It is easy to check that $w(T^2P+PT^2)\leq 2w(T^2)\|P\|, $ (see \cite{FH}) and so the bound obtained in Theorem \ref{theorem:upperbound} improves on the bound obtained  by Abu-Omar and Kittaneh \cite{OK}, namely,
\begin{eqnarray*}
  w^{4}(T)&\leq&\frac{1}{4}w^2(T^2)+\frac{1}{4}w(T^2)\|P\|+\frac{1}{16}\|P\|^2.
\end{eqnarray*}
Abu-Omar and Kittaneh \cite{OK} also proved that this bound is better than the bounds  obtained in \cite{FK2,FK} 
 \begin{eqnarray*}
  w(T)&\leq& \frac{1}{2}\big( \|T\|+\|T^2\|^{\frac{1}{2}}\big) 
	\end{eqnarray*} and
	\begin{eqnarray*}
	w^2(T)&\leq& \frac{1}{2}\| P\|.
\end{eqnarray*}
 Dragomir \cite{D} proved that  $w^2(T)\leq \frac{1}{2}[w(T^2)+\|T\|^2]$, i.e., $w^4(T)\leq \frac{1}{4}w^2(T^2)+\frac{1}{2}w(T^2)\|T\|^2+\frac{1}{4}\|T\|^4$ which is weaker than the bound obtained by Abu-Omar and Kittaneh \cite{OK}.
Thus the bound obtained here improves on all the existing upper bounds on numerical radius inequalities. 
\end{remark}
We next prove the following inequality. 
\begin{theorem}\label{theorem:upperboundA}
Let $T \in B(H)$. Then 
\[w^{3}(T)\leq \frac{1}{4}w(T^3)+\frac{1}{4}w(T^2T^*+{T^*}T^2+TT^*T).\] Moreover if $T^2=0$ then $w(T)=\frac{1}{2}\sqrt{\|TT^*+{T^*}T\|}$ and if $T^3=0$ then $w(T)=\big(\frac{1}{4}w(T^2T{^*}+{T{^*}}T^2+TT{^*}T)\big)^{\frac{1}{3}}$.
\end{theorem}
\begin{proof}

We note that $w(T)=\sup_{\theta \in \mathbb{R}}\|H_\theta\|$ where $H_{\theta}=\textit{Re}(e^{i \theta} T)$. Then,
\begin{eqnarray*}
    H_{\theta} & = & \frac{1}{2}(e^{i\theta} T + e^{-i\theta} T^*) \\
   \Rightarrow 4 {H_{\theta}}^2  &= & e^{2i\theta} T^2 + e^{-2i\theta} {T^*}^2 + T^{*}T+TT^{*} \\ 
\Rightarrow 8H^3_{\theta}  &= &\big( e^{2i\theta} T^2 + e^{-2i\theta} {T^*}^2 + T^{*}T+TT^{*}\big)(e^{i\theta} T + e^{-i\theta} T^*)\\
\Rightarrow H^3_{\theta}  &= & \frac{1}{4}\textit{Re}(e^{3i\theta} T^3) +\frac{1}{4}\textit{Re}(e^{i\theta} (T^2T{^*}+{T{^*}}T^2+TT{^*}T)\\
\Rightarrow \|H^3_{\theta}\|  &\leq & \frac{1}{4}\|\textit{Re}(e^{3i\theta} T^3)\| +\frac{1}{4}\|\textit{Re}(e^{i\theta} (T^2T{^*}+{T{^*}}T^2+TT{^*}T)\|.
\end{eqnarray*}
Taking the supremum over $\theta \in \mathbb{R}$ in the above inequality we have the desired inequality.
If $T^2=0$ then $4 {H_{\theta}}^2  =  T^{*}T+TT^{*}$ and so $w(T)=\frac{1}{2}\sqrt{\|TT^*+{T^*}T\|}.$
\smallskip
If $T^3=0$ then $H^3_{\theta} = \frac{1}{4}\textit{Re}(e^{i\theta} (T^2T{^*}+{T{^*}}T^2+TT{^*}T)$ and so $w^3(T)=\frac{1}{4}w(T^2T{^*}+{T{^*}}T^2+TT{^*}T)$.
\end{proof}

\begin{remark}

Abu-Omar and kittaneh \cite{OK}  proved that $w^2(T)\leq \frac{1}{2}w(T^2)+\frac{1}{4}\|TT^*+T^*T\|$. Our inequality obtained  in Theorem \ref{theorem:upperboundA} gives a better bound for the numerical radius for the  matrix $T$ than the bound obtained in \cite{OK}, where
$T=\left(\begin{array}{ccc}
    1&1&2 \\
    0&-1&1\\
		0&0&0
 \end{array}\right).$ In particular,  $w(T)\leq 1.784$ if we follow the inequality obtained in Theorem \ref{theorem:upperboundA}, whereas $w(T)\leq 1.989$  if we follow the bound obtained in \cite{OK}.
\end{remark}
We next prove the following inequality.
 
\begin{theorem}\label{theorem:upperbound1}
Let $T \in B(H)$. Then for each $r\geq 1$, 
\[w^{2r}(T)\leq \frac{1}{2}w^r(T^2)+\frac{1}{4}\big\|(T^{*}T)^r+(TT^{*})^r\big\|.\] 
\end{theorem}
\begin{proof}
We note that $w(T)=\sup_{\theta \in \mathbb{R}}\|H_\theta\|$ where $H_{\theta}=\textit{Re}(e^{i \theta} T)$. Now,
\begin{eqnarray*}
    H_{\theta} & = & \frac{1}{2}(e^{i\theta} T + e^{-i\theta} T^*) \\
   \Rightarrow 4 {H_{\theta}}^2  &= & e^{2i\theta} T^2 + e^{-2i\theta} {T^*}^2 + T^{*}T+TT^{*} \\ 
   \Rightarrow {H_{\theta}}^2 & = & \frac{1}{2}\textit{Re}(e^{2i\theta} T^2) +\frac{1}{4}(T^{*}T+TT^{*})\\
	 \Rightarrow \|{H_{\theta}}^2\| & \leq & \frac{1}{2} \big\|\textit{Re}(e^{2i\theta} T^2)\big\| +\frac{1}{4}\big\|T^{*}T+TT^{*}\big\| 
\end{eqnarray*}
For $r\geq 1$,  $t^r$ and $t^{\frac{1}{r}}$ are convex and  concave operator functions respectively and using that we get,
\begin{eqnarray*}
\|{H_{\theta}}^2\|^r &\leq & \big \{\frac{1}{2} \big\|\textit{Re}(e^{2i\theta} T^2)\big\| +\frac{1}{2}\big\|\frac{T^{*}T+TT^{*}}{2}\big\| \big\}^r\\
&\leq & \frac{1}{2} \big\|\textit{Re}(e^{2i\theta} T^2)\big\|^r +\frac{1}{2}\big\|\frac{T^{*}T+TT^{*}}{2}\big\|^r\\
&\leq & \frac{1}{2} \big\|\textit{Re}(e^{2i\theta} T^2)\big\|^r +\frac{1}{2}\big\|(\frac{(T^{*}T)^r+(TT^{*})^r}{2})^{\frac{1}{r}}\big\|^r\\
&= & \frac{1}{2} \big\|\textit{Re}(e^{2i\theta} T^2)\big\|^r +\frac{1}{2}\big\|\frac{(T^{*}T)^r+(TT^{*})^r}{2}\big\|
\end{eqnarray*}
Now taking the supremum over $\theta \in \mathbb{R}$ in the above inequality we get,
\begin{eqnarray*}
  w^{2r}(T)&\leq&\frac{1}{2}w^r(T^2)+\frac{1}{4}\big\|(T^{*}T)^r+(TT^{*})^r\big\|.
\end{eqnarray*}

\end{proof}

\begin{remark}
For $A, B \in B(H)$, Sattari et. al. \cite{SMY} proved that $w^r(B^*A)\leq \frac{1}{4}\|(AA^{*})^r+(BB^*)^r\|+\frac{1}{2}w^r(AB^*)$. When $A=B^*$ then $w^r(A^2)\leq \frac{1}{4}\|(AA^{*})^r+(A^*A)^r\|+\frac{1}{2}w^r(A^2).$ Thus for the case $A=B^*$ our bound obtained in theorem \ref{theorem:upperbound1} is better than the bound obtained by Sattari et. al. \cite{SMY}.  
\end{remark}

Next we give  another upper bound for the numerical radius $w(T)$ in terms of $\|H_\phi\|.$ 

\begin{theorem}\label{theorem:upperbound2}
Let $T \in B(H)$. Then \[w(T) \leq \inf_{\phi \in \mathbb{R}}\sqrt{{\|H_\phi\|}^2+{\|H_{\phi+\frac{\pi}{2}}\|}^2}\] where $H_{\phi}=\textit{Re}(e^{i \phi} T)$.
\end{theorem}
\begin{proof}
 We have, $H_{\theta}=\textit{Re}(e^{i \theta} T)=\cos\theta \textit{Re}(T)-\sin\theta \textit{Im}(T)$. Then for  $\phi \in [0,2\pi],$ we get 
\begin{eqnarray*}
 H_{\theta+\phi} &=& \cos(\theta+\phi)\textit{Re}(T)-\sin(\theta+\phi)\textit{Im}(T)\\
   &=& \cos\theta [\cos\phi \textit{Re}(T)-\sin\phi \textit{Im}(T)]-\sin\theta [\sin\phi \textit{Re}(T)+\cos\phi \textit{Im}(T)]\\
	 &=& \cos\theta [\cos\phi \textit{Re}(T)-\sin\phi \textit{Im}(T)]-\sin\theta [-\cos(\phi+\frac{\pi}{2}) \textit{Re}(T) \\
	 & + & \sin(\phi+\frac{\pi}{2}) \textit{Im}(T)]\\
	&=&\cos\theta \textit{Re}(e^{i\phi}T)+\sin\theta \textit{Re}(e^{i(\phi+\frac{\pi}{2})}T)\\
	 &=& H_\phi \cos\theta + H_{\phi+\frac{\pi}{2}} \sin\theta \\
	\Rightarrow \|H_{\theta+\phi}\| &\leq &\|H_\phi\cos\theta \|+\|H_{\phi+\frac{\pi}{2}}\sin\theta \|\\
	 \Rightarrow \|H_{\theta+\phi}\|  &\leq & \sqrt{{\| H_\phi\|}^2+{\| H_{\phi+\frac{\pi}{2}}\|}^2}.
\end{eqnarray*}
Taking supremum over $\theta \in \mathbb{R}$ in the above inequality, we get
\[ w(T) \leq  \sqrt{{\| H_\phi\|}^2+{\| H_{\phi+\frac{\pi}{2}}\|}^2}.\]
This is true for any $\phi \in \mathbb{R}$ and so we get,
\[w(T) \leq \inf_{\phi \in \mathbb{R}}\sqrt{{\|H_\phi\|}^2+{\|H_{\phi+\frac{\pi}{2}}\|}^2}.\]
 \end{proof}

\begin{remark}\label{remark:remark1}
Noting that for $ \phi = 0, \| H_\phi \| = \| Re(T) \| $ and $\| H_{\phi + \pi/2}\| = \| Im(T) \|,$  it follows from Theorem \ref{theorem:upperbound2} that  $w(T) \leq \sqrt{\|\textit{Re}(T)\|^2+\|\textit{Im}(T)\|^2}$. Also, this inequality follows directly from the definition of the numerical radius by considering the Cartesian decomposition of $T$.
\end{remark}

Next we give an upper bound for the numerical radius of $n\times n$ operator matrices which follows from \cite[Theorem $2$]{OK2} and \cite[Remark $1$] {OK2}. 
\begin{theorem}\label{theorem:upperbound3}
Let $H_1, H_2,\ldots,H_n$ be Hilbert spaces and $H=\bigoplus ^n_{i=1}H_i$. If $A=(A_{ij})$ be an $n\times n$ operator matrix acting on $H$ with $A_{ij}\in B(H_j,H_i)$, then
\[w(A)\leq \max_{1\leq i \leq n}\big\{ w(A_{ii})+\frac{1}{2}\sum^n_{j=1,j\neq i}(\|A_{ij}\|+\|A_{ji}\|)\big\}.\]
\end{theorem}

Using above Theorem  \ref{theorem:upperbound3} we can estimate the spectral radius of sum of product of $n$ pairs of operators as follows.
\begin{theorem}\label{theorem:spectral} 
Let $A_i, B_i \in B(H)$. The spectral radius of $\sum^n_{i=1}A_iB_i$ satisfies the inequality
\[\rho (\sum^n_{i=1}A_iB_i)\leq \max_{1\leq i\leq n}\{ w(B_iA_i)+\frac{1}{2}\sum^n_{j=1,j\neq i}(\|B_iA_j\|+\|B_jA_i\|)\}.\]

\end{theorem}
\begin{proof}
We have 
\begin{eqnarray*}
\rho (\sum^n_{i=1}A_iB_i)&=& \rho \left(\begin{array}{cccccc}
    \sum^n_{i=1}A_iB_i&0&.&.&.&0 \\
    0&0&.&.&.&0 \\
    .& & & & &  \\
    .& & & & & \\
    .& & & & & \\
    0&0&.&.&.&0
    \end{array}\right) \\
		&=&\rho \left(\begin{array}{cccccc}
    A_1&A_2&.&.&.&A_n \\
    0&0&.&.&.&0 \\
    .& & & & &  \\
    .& & & & & \\
    .& & & & & \\
    0&0&.&.&.&0
    \end{array}\right)\left(\begin{array}{cccccc}
    B_1&0&.&.&.&0 \\
    B_2&0&.&.&.&0 \\
    .& & & & &  \\
    .& & & & & \\
    .& & & & & \\
    B_n&0&.&.&.&0
    \end{array}\right)\\
		&=& \rho \left(\begin{array}{cccccc}
    B_1&0&.&.&.&0 \\
    B_2&0&.&.&.&0 \\
    .& & & & &  \\
    .& & & & & \\
    .& & & & & \\
    B_n&0&.&.&.&0
    \end{array}\right)\left(\begin{array}{cccccc}
    A_1&A_2&.&.&.&A_n \\
    0&0&.&.&.&0 \\
    .& & & & &  \\
    .& & & & & \\
    .& & & & & \\
    0&0&.&.&.&0
    \end{array}\right)\\
		&=& \rho \left(\begin{array}{cccccc}
    B_1A_1&B_1A_2&.&.&.&B_1A_n \\
    B_2A_1&B_2A_2&.&.&.&B_2A_n\\
    .& & & & &  \\
    .& & & & & \\
    .& & & & & \\
    B_nA_1&B_nA_2&.&.&.&B_nA_n
    \end{array}\right)
		\end{eqnarray*}
		\begin{eqnarray*}
		&\leq& w\left(\begin{array}{cccccc}
    B_1A_1&B_1A_2&.&.&.&B_1A_n \\
    B_2A_1&B_2A_2&.&.&.&B_2A_n\\
    .& & & & &  \\
    .& & & & & \\
    .& & & & & \\
    B_nA_1&B_nA_2&.&.&.&B_nA_n
    \end{array}\right)\\
		&\leq& \max_{1\leq i\leq n}\big\{ w(B_iA_i)+\frac{1}{2}\sum^n_{j=1,j\neq i}(\|B_iA_j\|+\|B_jA_i\|)\big\}.
\end{eqnarray*}
\end{proof}

\section{On lower bound of numerical radius inequalities}
We begin this section with following inequality on lower bound of numerical radius.
\begin{theorem}\label{theorem:lowerbound1}
Let $T \in B(H)$. Then  
\[w^{4}(T)\geq \frac{1}{4}C^2(T^2)+\frac{1}{8}m(T^2P+PT^2)+\frac{1}{16}\|P\|^2, \]  
\noindent where $ P=T^{*}T+TT^{*}, C(T)=\inf_{x\in H,\|x\|=1}\inf_{\phi \in \mathbb{R}}\|\textit{Re}(e^{i\phi} T)x\|.$
\end{theorem}

\begin{proof}
We know that $w(T)=\sup_{\phi \in \mathbb{R}}\|H_\phi\|$ where $H_{\phi}=\textit{Re}(e^{i \phi} T)$. Let $x$ be a unit vector in $H$ and let $\theta$ be a real number such that $e^{2i\theta}\langle(T^2P+PT^2)x,x\rangle=|\langle(T^2P+PT^2)x,x\rangle|.$ Then
\begin{eqnarray*}
    H_{\theta} & = & \frac{1}{2}(e^{i\theta} T + e^{-i\theta} T^*) \\
   \Rightarrow 4 {H_{\theta}}^2  &= & e^{2i\theta} T^2 + e^{-2i\theta} {T^*}^2 + P \\ 
	\Rightarrow 16 {H_{\theta}}^4  &= & \big(e^{2i\theta} T^2 + e^{-2i\theta} {T^*}^2 + P \big) \big(e^{2i\theta} T^2 + e^{-2i\theta} {T^*}^2 + P \big)\\
	       &= & \big(e^{2i\theta} T^2 + e^{-2i\theta} {T^*}^2 \big)^2+\big(e^{2i\theta} T^2 + e^{-2i\theta} {T^*}^2 \big)P \\ &+& P \big(e^{2i\theta} T^2 + e^{-2i\theta} {T^*}^2 \big)+P^2\\
& = & 4\big(\textit{Re}(e^{2i\theta} T^2)\big)^2+ 2 \textit{Re}(e^{2i\theta} (T^2P+PT^2))  +P^2\\
\Rightarrow 16 w^4(T) &\geq&  \|4\big(\textit{Re}(e^{2i\theta} T^2)\big)^2+ 2 \textit{Re}(e^{2i\theta} (T^2P+PT^2))  +P^2\|\\
	&\geq&  | \langle \big(4\big(\textit{Re}(e^{2i\theta} T^2)\big)^2+ 2 \textit{Re}(e^{2i\theta} (T^2P+PT^2))  +P^2\big)x,x \rangle| \\
	&=& \mid 4 \langle\big(\textit{Re}(e^{2i\theta} T^2)\big)^2x,x \rangle + 2 \textit{Re}(e^{2i\theta} \langle (T^2P+PT^2)x,x\rangle)  +\langle P^2x,x \rangle | \\
	&=&  4 \|\big(\textit{Re}(e^{2i\theta} T^2)\big)x \|^2+ 2 |\langle (T^2P+PT^2)x,x\rangle|  +\|Px\|^2 \\
	&\geq&  4 C^2(T^2)+ 2 m(T^2P+PT^2)  +\|Px\|^2 \\
\Rightarrow 16 w^4(T) &\geq& 4 C^2(T^2)+ 2 m(T^2P+PT^2)  +\sup_{\|x\|=1}\|Px\|^2 \\
	&=&  4 C^2(T^2)+ 2 m(T^2P+PT^2)  +\|P\|^2 \\
\Rightarrow w^{4}(T)&\geq& \frac{1}{4}C^2(T^2)+\frac{1}{8}m(T^2P+PT^2)+\frac{1}{16}\|P\|^2.	
\end{eqnarray*}
 This completes the proof. 
\end{proof}
\begin{remark}
Kittaneh\cite{FK} proved that $w^2(T)\geq \frac{1}{4}\|P\|,$ which easily follows from our  Theorem \ref{theorem:lowerbound1}. 
\end{remark}
 We next prove the following inequalities involving $\textit{Re}(T)$ and $\textit{Im}(T)$.
\begin{theorem}\label{theorem:lowerbound4}
Let $T\in B(H)$. Then $w(T)\geq \sqrt{\|\textit{Re}(T)\|^2+ m^2(\textit{Im}(T))}$ and $w(T)\geq \sqrt{\|\textit{Im}(T)\|^2+ m^2(\textit{Re}(T))}$.
\end{theorem}
\begin{proof}
First we assume $\|\textit{Re}(T)\|=|\lambda|$. Therefore, there exists a sequence $\{x_n\}$ in $H$ with $\|x_n\|=1$ such that $\langle \textit{Re}(T)x_n,x_n\rangle \rightarrow \lambda$. Now 
\begin{eqnarray*}
\langle Tx_n,x_n \rangle &=&\langle (\textit{Re}(T)+i\textit{Im}(T))x_n,x_n \rangle \\
\Rightarrow\langle Tx_n,x_n \rangle&=&\langle \textit{Re}(T)x_n,x_n \rangle +i \langle \textit{Im}(T)x_n,x_n \rangle \\
\Rightarrow |\langle Tx_n,x_n \rangle |^2 &=& (\langle \textit{Re}(T)x_n,x_n \rangle)^2 + (\langle \textit{Im}(T)x_n,x_n \rangle)^2\\
\Rightarrow |\langle Tx_n,x_n \rangle |^2&\geq& (\langle \textit{Re}(T)x_n,x_n \rangle)^2 +m^2(\textit{Im}(T))\\
\Rightarrow {w^2(T)}&\geq& \lambda^2 +m^2(\textit{Im}(T))\\
\Rightarrow {w(T)}&\geq& \sqrt {\|\textit{Re}(T)\|^2 +m^2(\textit{Im}(T))}.
\end{eqnarray*}
The proof of other inequality follows in the same way.
\end{proof}

Note that if $\textit{Re}(T)$  and $\textit{Im}(T)$  are unitarily equivalent to scalar operators then $\|\textit{Re}(T)\|=m(\textit{Re}(T))$ and  $\|\textit{Im}(T)\|=m(\textit{Im}(T))$ respectively. Therefore from Remark \ref{remark:remark1} and Theorem \ref{theorem:lowerbound4}  we get the following equality.
\begin{cor}
Let $T \in B(H)$. If either $\textit{Re}(T)$ or $\textit{Im}(T)$ is unitarily equivalent to a scalar operator then $w(T)=\sqrt{\|\textit{Re}(T)\|^2+ \|\textit{Im}(T) \|^2}$.
\end{cor}
\begin{remark}
 For $T\in B(H)$, Kittaneh et. al. \cite{KMY}  proved that  $w(T)\geq \|\textit{Re}(T)\|$ and  $w(T)\geq \|\textit{Im}(T)\|$. For any bounded linear operators these bounds are weaker than the bounds obtained in Theorem \ref{theorem:lowerbound4}.
\end{remark}

\section {Estimation of zeros of polynomial}
As an application of the inequalities obtained in the previous section we can estimate zeros of the polynomial. Let $p(z)=z^n+a_{n-1}z^{n-1}+\ldots+a_1z+a_0$ be a monic polynomial of degree $n\geq 2$ with complex coefficients $a_0, a_1, \ldots, a_{n-1}$. Then the Frobenius companion matrix of $p$ is given by  
\begin{eqnarray*}
  C(p)&=&\left(\begin{array}{ccccccc}
    -a_{n-1}&-a_{n-2}&.&.&.&-a_1&-a_0 \\
    1&0&.&.&.&0&0\\
		0&1&.&.&.&0&0\\
    .& & & & & & \\
    .& & & & & &\\
    .& & & & & &\\
    0&0&.&.&.&1&0
    \end{array}\right).
		\end{eqnarray*}
		Then the eigenvalues of $C(p)$ are exactly the zeros of the polynomial $p(z)$. Considering $C(p)$ as a linear operator on $\mathbb{C}^n$, we see that if $z$ is a zero of the polynomial  $p(z)$ then $|z|\leq w(C(p))$ as $\sigma(C(p))\subseteq \overline {W(C(p))}$. Many mathematicians have estimated zeros of the polynomial using this approach, some of them are mentioned below. Let $\mu$ be a zero of the polynomial  $p(z)$.\\
	(1) Carmichael and Mason \cite{CM} proved that 
		\[ |\mu|\leq (1+|a_0|^2+|a_1|^2+\ldots+|a_{n-1}|^2)^{\frac{1}{2}}.\]
	(2) Cauchy \cite{CM}	proved that 
		   \[ |\mu|\leq 1+\max \{ |a_0|, |a_1|, \ldots, |a_{n-1}|\}.\]
	(3) Fujii and Kubo \cite{FK1} proved that 
			\[|\mu|\leq \cos\frac{\pi}{n+1}+\frac{1}{2}\big[\big(\sum_{j=0}^{n-2}|a_j|^2\big)^{\frac{1}{2}}+|a_{n-1}|\big].\]
	(4) Kittaneh \cite{FK2} proved that 
			\[ |\mu| \leq\frac{1}{2}\big[ \|C(p)\|+\| C(p)^2\|^{\frac{1}{2}}\big].\]
	(5) Paul and Bag \cite{PB1} proved that 
		   \[ |\mu| \leq\frac{1}{2}\big[ |a_{n-1}|+\cos\frac{\pi}{n}+\sqrt{ (|a_{n-1}|-\cos\frac{\pi}{n})^2+(1+\sqrt{\sum_{k=2}^n|a_{n-k}|^2})^2}\big].\]
	(6) Paul and Bag \cite{PB2} proved that 
	\[ |\mu| \leq\frac{1}{2}\big[w(A)+\cos\frac{\pi}{n-1}+\sqrt{(w(A)-\cos\frac{\pi}{n-1})^2+(1+\sqrt{\sum_{k=3}^n|a_{n-k}|^2})^2}\big],\]where
$A=\left(\begin{array}{cc}
    -a_{n-1}&-a_{n-2} \\
    1&0
 \end{array}\right).$\\
	(7)   Abu-Omar and Kittaneh \cite{A} proved that 
	     \[|\mu| \leq \sqrt{\frac{1}{4}(|a_{n-1}|^2+\alpha)^2+\alpha+\cos^2\frac{\pi}{n+1}}, \] where $\alpha=\sqrt{\sum_{j=0}^{n-1}|a_j|^2}$.\\
	(8) Alpin et. al. \cite{YML} proved that
	     \[ |\mu| \leq  \max_{1\leq k \leq n}[(1+|a_{n-1}|)(1+|a_{n-2}|)\ldots(1+|a_{n-k}|)]^{\frac{1}{k}}.\]

		Using Theorem \ref{theorem:upperbound1} and observing that spectral radius is always dominated by numerical radius we can easily prove the following theorem.
			\begin{theorem}\label{theorem:zero}
		If $\mu$ is a zero of the polynomial $p(z)$, then 
		\[ |\mu| \leq \big(\frac{1}{2} w^2(C^2)+ \frac{1}{4}\|(C^*C)^2+(CC^*)^2\| \big )^{\frac{1}{4}},\] where $C=C(p)$.
	\end{theorem}	
	In similar way, using Theorem \ref{theorem:upperbound} we have the following theorem.
	\begin{theorem}\label{theorem:zero1}
		If $\mu$ is a zero of the polynomial $p(z)$, then 
		\[ |\mu| \leq \big(\frac{1}{4}w^2(C^2)+\frac{1}{8}w(C^2P+PC^2)+\frac{1}{16}\|P\|^2\big)^{\frac{1}{4}},\]  where $ C=C(p), P=C^{*}C+CC^{*}.$
\end{theorem}
		
We illustrate with an example to show that the above bounds obtained by us is better than the existing bounds. 
	\begin{example}
	Consider the polynomial $p(z)=z^5+z^4-2$. Then the upper bounds of zeros of this polynomial $p(z)$ estimated by different mathematicians are as shown in the following table.
	\end{example}
	\begin{center}
\begin{tabular}{ |c|c| } 
 \hline
  Carmichael and Mason \cite{CM}& 2.449 \\
 \hline
    Cauchy \cite{CM} &  3.000 \\ 
 \hline 
   Fujii and Kubo \cite{FK1} & 2.366  \\
 \hline 
Kittaneh \cite{FK2} & 2.085\\
\hline
	Alpin et. al. \cite{YML} & 2.000 \\
\hline 
 Paul and Bag \cite{PB1} & 2.407 \\
\hline
Paul and Bag \cite{PB2} & 2.477 \\
\hline
Abu-Omar and Kittaneh \cite{A} & 2.367 \\
\hline
 \end{tabular}
\end{center}
But if $\mu$ is a zero of the polynomial  $p(z)=z^5+z^4-2$ then Theorem \ref{theorem:zero} gives $|\mu|\leq 1.692$ and Theorem \ref{theorem:zero1} gives  $|\mu|\leq 1.881$ which are better than all the estimations mentioned above.

\bibliographystyle{amsplain}

\end{document}